\documentclass[oneside,reqno]{amsart}

\usepackage{amsmath}
\usepackage{amssymb}
\usepackage{amsfonts}   
\usepackage{booktabs}    
\usepackage{nicefrac}       
\usepackage{microtype}      
\usepackage{bm}
\usepackage{dsfont}
\usepackage{color,soul}
\usepackage[round]{natbib}
\usepackage{tabularx}
\usepackage{graphicx}
\usepackage{caption}
\usepackage{subcaption}
\usepackage[margin=3cm]{geometry}
\usepackage[onehalfspacing]{setspace}
\usepackage[hidelinks]{hyperref}  	    
\usepackage{multirow}

\newtheorem{theorem}{Theorem}[section]
\newtheorem{lemma}[theorem]{Lemma}
\newtheorem{corollary}[theorem]{Corollary}
\theoremstyle{definition}

\newtheorem{example}{Example}[section]

\theoremstyle{remark}
\newtheorem{remark}{Remark}[section]

\newcommand{\R}{\mathbb{R}}
\newcommand{\XX}{\mathcal{X}}

\newcommand{\E}{\mathbb{E}}

\DeclareMathOperator*{\argmax}{arg\,max}

\newcommand*\diff{\mathop{}\!\mathrm{d}}

\newcommand{\1}{\mathds{1}}

\DeclareMathOperator{\eff}{eff}

\newcommand{\B}{\boldsymbol{\beta}}
\newcommand{\f}{{\bf f}}
\newcommand{\M}{{\bf M}}

\renewcommand{\E}{\operatorname{E}}
\newcommand{\Prob}{\operatorname{P}}
\newcommand{\trace}{\operatorname{trace}}

\numberwithin{equation}{section}

\begin{document}
	
	\title{Poisson Regression in one Covariate on Massive Data}
	
	
	\author{Torsten Reuter}
	\thanks{Corresponding author: Torsten Reuter. \textit{E-mail address}: \texttt{torsten.reuter@ovgu.de}}
	\address{Torsten Reuter. Otto von Guericke University Magdeburg. Universitätsplatz 2, 39106 Magdeburg,
		Germany}
	\curraddr{}
	\email{torsten.reuter@ovgu.de}
	
	\author{Rainer Schwabe}
	\address{Rainer Schwabe. Otto von Guericke University Magdeburg. Universitätsplatz 2, 39106 Magdeburg,
		Germany}
	\curraddr{}
	\email{rainer.schwabe@ovgu.de}
	
	\date{}
	
	
\begin{abstract}
	The goal of subsampling is to select an informative subset of all observations, when using the full data for statistical analysis is not viable.
	We construct locally $ D $-optimal subsampling designs under a Poisson regression model with a log link in one covariate.
	A Representation of the support of locally $ D $-optimal subsampling designs is established. 
	We make statements on scale-location transformations of the covariate that require a simultaneous transformation of the regression parameter. 
	The performance of the methods is demonstrated by illustrating examples.
	To show the advantage of the optimal subsampling designs, we examine the efficiency of uniform random subsampling as well as of two heuristic designs. 
	Further, the efficiency of locally $ D $-optimal subsampling designs is studied when the parameter is misspecified.
\end{abstract}
	
\maketitle
\section{Introduction}\label{sec:intro}
Progress in technology has lead to the collection of increasingly large data sets. 
The field of subsampling or subdata selection has gained popularity in recent years, where the aim is to decrease the number of observations in the data set while maintaining as much information as possible. 
To illuminate fundamental features of the concept, we solely focus on the reduction of observations in massive data for a single covariate, rather than reduction in covariates of high-dimensional data.
Subdata selection for massive data can be done via a probabilistic subsampling scheme or through deterministic rules.    
Earlier works on subsampling for generalized linear models (GLMs) focus on probabilistic methods, in particular on subsampling for logistic regression, see e.g. \cite{wang2018optimal}.
More recently there are more works on GLMs, including Poisson regression: 
For probabilistic subsampling under the $ A $ and $ L $-optimality criteria see \cite{ai2021optimal} and \cite{yu2022optimal}.
After \cite{wang2019information} introduced information-based optimal subdata selection (IBOSS) for linear regression, \cite{cheng2020information} proposed IBOSS for logistic regression, a deterministic subsampling technique with a probabilistic initial subsample to estimate the unknown parameter. 
This is necessary because, as is well known, the optimal design depends on the unknown parameter for GLMs. 

In the present paper on Poisson regression we derive locally $ D $-optimal continuous subsampling designs directly bounded by the density of the covariate. 
Such directly bounded designs were first studied by \cite{wynn1977optimum} and \cite{fedorov1989optimal}.
Recently, \cite{ul2019optimal} derived such bounded optimal subsampling designs for logistic regression in the context of optimal item calibration similarly to our approach. 
Such subsampling designs can then easily be used for subdata selection by including all observations that lie in the support of the optimal subsampling design and exclude all others. 
Though an initial step to estimate the parameter is necessary when it is unknown. 
When there are no constraints on the design, literature on Poisson regression includes \cite{rodriguez2007locally} and \cite{russell2009d}.

In Section~\ref{sec:model} we introduce the Poisson regression model to be used in this paper.
Then, we present a theorem on the support of a locally $ D $-optimal continuous subsampling design as well as a theorem concerning scale-location shifts of the covariate in Section~\ref{sec:design}. 
Further, we give examples when the covariate has an exponential or a uniform distribution. 
In Section~\ref{sec:efficiency} we study the efficiency of uniform random subsampling and some heuristic designs in comparison to the optimal subsampling designs. 
In addition, we consider the loss in efficiency when the regression parameter is misspecified. 
We add closing remarks in Section~\ref{sec:discussion}.
Proofs are deferred to an appendix.

\section{Model Specification}
\label{sec:model}
We consider pairs $ (x_{i}, y_{i}), i = 1,\dots,n $, of data, 
where $ y_{i} $ is the value of the response variable $ Y_{i} $. 
$x_i$ is a realization of the random variable $ X_{i} $. 
The covariate $ X_{i} $ has probability density function $ f_{X} $.
We suppose that the dependence of the response variable on the covariate $ X_{i} $ is given by a Poisson regression model.
\begin{itemize}
	\item[(A1)]
	Conditionally on the covariate $ X_i $, the response $ Y_i $ is Poisson distributed with conditional mean
	$\E(Y_{i}|X_{i}) = \exp(\beta_{0} + \beta_{1}X_{i})$.
\end{itemize}
Model (A1) constitutes a generalized linear model with random covariate and log link. 
The aim is to estimate the regression parameter 
$ \B = (\beta_0 ,\beta_1)^{\top} $.
$ \f(x) = (1,x)^{\top} $ denotes the regression function 
in the linear component $\f(X_{i})^{\top}\B$
such that 
$ \E(Y_{i}|X_{i}) = \exp(\f(X_{i})^{\top}\B)$.

We will further assume that the covariate $X_{i}$ has a continuous distribution satisfying some moment conditions. 
\begin{itemize}
	\item[(A2)]
	The covariate $X_{i}$ has density $f_{X}$ and $ \E(X_i^2 \exp(\beta_{1} X_{i})) < \infty $.
\end{itemize}

\section{Subsampling Design}
\label{sec:design}
We assume that the number of observations $ n $ is very large.  
However, we encounter the challenge of dealing with responses, denoted by $Y_{i}$, which are either costly or difficult to observe. 
Meanwhile, the values $x_i$ of all units $X_i$ of the covariate are readily available. 
To tackle this problem, we consider a scenario in which the responses $Y_{i}$ will only be observed for a specific subsampling proportion $\alpha$ of the units, $0 < \alpha < 1$. 
The selection of these units is based on the knowledge of the covariate values $x_{i}$ for all units. 
Our objective is to identify a subsample of pairs $(x_{i}, y_{i})$ that provides the most accurate estimation of the parameter vector $\B$ by means of the maximum likelihood estimator $ \hat{\B} $.
As the covariate $X_i$ has a continuous distribution, we are going to identify a subsample from this distribution that maximizes information, but only covers a percentage $\alpha$ of the distribution.
Therefore, we consider continuous designs $ \xi $ 
as measures of mass $\alpha$ on $\R$ with density $f_{\xi}$ bounded by the density $f_X$ 
of $X_i$ ensuring $ \int f_{\xi}(x)\diff x = \alpha $ and
$  f_{\xi}(x) \leq f_X(x)$ for all $x \in \R$. 
A subsample can then be generated according to such a bounded continuous design $ \xi $
by accepting units $i$ with probability $f_\xi(x_i) / f_X(x_i)$.
To obtain analytical results, we assume that the distribution of the covariate $X_{i}$ and, hence, its density $f_{X}$ is known. 

The information arising for a single observation at covariate value $ x $ is defined by the elemental information 
$ \M(x, \B) = \exp(\beta_0 + \beta_1 x) \f(x) \f(x)^\top $ \cite[see][]{russell2009d}.
For a continuous design $\xi$, the information matrix $ \M(\xi,\B)$ is defined by  
\begin{equation*}
	\M(\xi, \B) 
	= \int \M(x, \B) \xi(\diff x)
	= \exp(\beta_0) 
	\begin{pmatrix}
		m_{0}(\xi, \beta_{1})     & m_{1}(\xi, \beta_{1})      \\
		m_{1}(\xi, \beta_{1})          & m_{2}(\xi, \beta_{1})     
	\end{pmatrix},
\end{equation*}
where $ m_{k}(\xi, \beta_{1}) = \int x^{k} \exp(\beta_{1} x) f_{\xi}(x) \diff x$.
The moment condition $ \E(X_i^2 \exp(\beta_{1} X_{i})) < \infty $ stated in assumption~(A2) for the distribution of the covariates $X_i$ ensures that the entries $m_{k}(\xi,\beta_{1})$ in the information matrix are finite for any bounded continuous design $ \xi $.
Otherwise no meaningful optimization would be possible.
The moment condition is obviously satisfied when the distribution of $ X_i $ has a finite support. It also holds for other not heavy-tailed distributions like the normal distribution. 
In the case of an exponentially distributed covariate $ X_i $ considered below, the additional condition $ \beta_1 < \lambda $ on the slope parameter $ \beta_1 $ is required where $ \lambda $ is the rate parameter of the exponential distribution.

The information matrix $\M(\xi,\B)$ serves as a measure for evaluating the performance of the design $\xi$.
Note that $\M(\xi, \B)$ has full rank for any continuous design $ \xi $.
This ensures the existence of the inverse 
\begin{align*}
	\label{eq:info-inverse-linear}
	\M(\xi,\B)^{-1} = \frac{1}{\exp(\beta_{0}) d(\xi, \beta_{1})}
	\begin{pmatrix}
		m_{2}(\xi, \beta_{1})     & - m_{1}(\xi, \beta_{1})      \\
		- m_{1}(\xi, \beta_{1})          & m_{0}(\xi,\beta_{1})     
	\end{pmatrix}.
\end{align*}
where $d(\xi, \beta_{1}) = m_{0}(\xi, \beta_{1}) m_{2}(\xi, \beta_{1}) - m_{1}(\xi, \beta_{1})^2$
is the standardized determinant of $\M(\xi,\B)$, $d(\xi, \beta_{1}) = \exp(- 2 \beta_{0}) \det(\M(\xi, \B))$.
Then, $\sqrt{\alpha n}(\hat\B - \B)$ is asymptotically normal with mean zero and covariance matrix $\M(\xi,\B)^{-1}$ for the maximum likelihood estimator $ \hat{\B} $.

Maximization of the information matrix in the Loewner sense of nonnegative definiteness will not be possible, in general.
Therefore, we have to consider some one-dimensional information functional.
We will focus here on the most popular design criterion,
the $D$-criterion, in its widely used form, $\log(\det(\M(\xi,\B)))$, to be maximized.
A subsampling design $\xi^{*}$ with density $ f_{\xi^{*}} $ that maximizes the $D$-criterion for a given parameter value $ \B $
will be called locally $D$-optimal at $ \B $.
Maximization of the $D$-criterion can be interpreted in terms of the
covariance matrix as minimization of the volume
of the asymptotic confidence ellipsoid for the parameter vector $\B$.
 
\begin{remark}
	\label{remark:beta0}
	Note that $ \beta_{0} $ comes in into the information matrix only by the multiplicative factor $\exp(\beta_{0})$. 
	Thus, a locally $D$-optimal subsampling design $\xi^*$ only depends on the slope $ \beta_{1} $.
\end{remark}

For the characterization of a locally $D$-optimal design,
we will make use of an equivalence theorem based on constrained 
convex optimization \citep[see e.\,g.][]{sahm2001note}.
For this, we have to distinguish between cases related to the sign of the slope $\beta_{1}$.
In applications, the slope will often be negative ($\beta_{1} < 0$).
We will focus on that case and establish a representation of the locally $ D $-optimal subsampling designs for $\beta_{1} < 0$ first. 

Denote by $F_X$ and $q_{\alpha}$ the cumulative distribution function and the $\alpha$-quantile of $X_i$.
Let $ \1_{A} $ the indicator function of a set $A$, i.\,e.~$ \1_{A}(x) = 1 $, if $ x \in A $ and $ \1_{A}(x) = 0 $ otherwise.
Further, denote by 
\[
	\psi(x, \xi, \beta_{1}) = \frac{1}{d(\xi, \beta_{1})} \exp(\beta_{1} x)
	(m_{0}(\xi, \beta_{1}) x^2 - 2 m_{1}(\xi, \beta_{1}) x + m_{2}(\xi,\beta_{1}))
\]
the sensitivity function of a design $\xi$ 
(see Theorem~\ref{theorem:opt-design}).
Note that the sensitivity function $\psi(x, \xi, \beta_{1})$
does not depend on $\beta_{0}$.

\begin{theorem}
	\label{theo:support}
	Let assumptions~\emph{(A1)} and \emph{(A2)} 
	be satisfied 
	and let $\beta_{1} < 0$. 
	Then the subsampling design $\xi^*$ 
	is locally $D$-optimal at $ \B $ 
	if and only if
	$\xi^*$ has density $f_{\xi^*}(x) = f_X(x) \1_{\XX^*}(x)$
	and either
	\begin{itemize}
	\item[(i)] 
	there exist $ a_1 < a_2 < a_3 $ such that
	\\ 
	$\XX^{*} = (- \infty, a_1] \cup [a_2, a_3]$, 
	\\
	$ F_X(a_1) + F_X(a_3) - F_X(a_2) = \alpha $, and
	\hfill
	\emph{(\ref{theo:support}a)}
	\\
	$\psi(a_1, \xi^{*}, \beta_{1}) = \psi(a_2, \xi^{*}, \beta_{1}) = \psi(a_3, \xi^{*}, \beta_{1})$,
	\hfill
	\emph{(\ref{theo:support}b)}
	\item[or]
	\item[(ii)]
	$ \XX^{*} = (- \infty, q_{\alpha}]$, 
	\hfill
	\emph{(\ref{theo:support}a')}
	\\
	$\psi(x, \xi^{*}, \beta_{1}) > \psi(q_{\alpha}, \xi^{*}, \beta_{1})$
	for $x < q_\alpha$, and
	$\psi(x,\xi^{*},\beta_{1}) < \psi(q_{\alpha}, \xi^{*}, \beta_{1})$
	for $x > q_\alpha$.
	\hfill
	\emph{(\ref{theo:support}b')}
	\end{itemize}
\end{theorem}

Conditions~(\ref{theo:support}a) and (\ref{theo:support}a')
correspond to  the subsampling percentage $\alpha$
while (\ref{theo:support}b) and (\ref{theo:support}b')
are related to the conditions on the sensitivity function
in the general equivalence theorem
for bounded designs  
(Theorem~\ref{theorem:opt-design})
reproduced in the Appendix.

In view of the shape $f_{\xi^*}(x) = f_X(x) \1_{\XX^*}(x)$ of the density of the continuous optimal subsampling designs $\xi^*$ in Theorem~\ref{theo:support}, the subsampling mechanism becomes deterministic for the optimal design:
The subsample can be generated 
by accepting all units $i$ for which $x_i \in \XX^*$ and by rejecting all others.

According to Theorem~\ref{theo:support}, there are two different scenarios for the locally $D$-optimal design $\xi^*$. Either the supporting set $\XX^*$ consists of two separate intervals $(- \infty, a_1]$ and $[a_2, a_3]$
(scenario~(i)) or these intervals will be merged into a single one (scenario~(ii)).

\begin{remark}
	\label{rem:unique}
	The optimal subsampling design $\xi^*$ is unique
	because of the strict concavity of the $D$-criterion
	and the shape of the sensitivity function.
\end{remark}

For the construction of a locally $ D $-optimal subsampling design by Theorem~\ref{theo:support}, 
first the conditions of scenario~(ii) for an optimal design supported on a single interval can be checked.
If scenario~(ii) does not apply,
the boundary points $ a_1 < a_2 < a_3 $ for the support $\XX^*$  have to be calculated by solving the system of (nonlinear) equations (\ref{theo:support}a) and (\ref{theo:support}b). 
In the latter case, the rightmost boundary point $ a_3 $ of $\XX^*$ may lie outside the support of $ X_{i} $, i.\,e.~$ a_3 > x_{\max}$, when the support of the covariate $X_{i}$ is bounded from above, i.\,e.~$x_{\max} = \mathrm{ess}\sup(X_i) < \infty$, where $\mathrm{ess}\sup$ denotes the essential supremum (see, e.\,g., Example~\ref{example:Unif} for the uniform distribution below). 
Then, in scenario~(ii), explicit calculation of the rightmost boundary point $ c $ is not necessary. Instead, it is sufficient for (\ref{theo:support}b) to verify that 
$ \psi(x_{\max}, \xi^{*}, \beta_{1}) \ge \psi(a_1, \xi^{*}, \beta_{1}) = \psi(a_2, \xi^{*}, \beta_{1}) $.
\begin{remark}
	\label{rem:left-interval}
	The leftmost boundary point $ a_1 $ of a $D$-optimal subsampling design $\xi^*$ cannot lie outside the range of $ X_{i} $, i.\,e.~$ a_1 > x_{\min} $, where $ x_{\min} = \mathrm{ess}\inf(X_i) $ the essential infimum of the distribution of $X_{i}$.
\end{remark}
\begin{remark}
	When $ \beta_{1} = 0 $, the information matrix $ \M(\xi,\B) $ is, up to the multiplicative constant $ \exp(\beta_{0}) $, equal to the information matrix $ \M(\xi) = \int \f(x) \f(x)^\top \xi(\diff x) $ in the linear model \cite[treated in][]{reuter2023optimal}. 
	Therefore, the $ D $-optimal subsampling design for ordinary linear regression is also  locally $ D $-optimal in the Poisson regression model. 
	Hence, according to~\cite[][Section~4]{reuter2023optimal}, 
	the subsampling design $\xi^*$ 
	is locally $D$-optimal for $ \beta_{1} = 0 $ 
	if and only if
	there exist $ a_1 < a_2 $ such that
	\\ 
	\hspace*{5mm}
	$f_{\xi^*}(x) = f_X(x) \1_{(- \infty, a_1] \cup [a_2, \infty)}(x)$,
	\\
	\hspace*{5mm}
	$ F_X(a_2) - F_X(a_1) = 1 - \alpha $, and
	\\
	\hspace*{5mm}
	$\psi(a_1, \xi^{*}, \beta_{1}) = \psi(a_2, \xi^{*}, \beta_{1})$.	
\end{remark}

By means of equivariance considerations, we may transfer a locally $ D $-optimal subsampling design $\xi^{*}$ for a covariate $ X_{i} $ to a location-scale transformed covariate $ Z_i =  a X_i + b $, $ a \neq 0 $.
However, the transformation of a locally $ D $-optimal subsampling design is not as straightforward as in polynomial regression \cite[see][]{reuter2023optimal}, but requires a simultaneous transformation of the slope parameter $ \beta_{1} $. 
This kind of simultaneous transformation typically has to be used in generalizes linear models where the elemental information depends on $\beta_{1}$ by the linear component $\f(x^{\top}) \beta_{1}$, see e.\,g.~\cite{radloff2016invariance}.
\begin{theorem}
	\label{theorem:equivariant}
	Let $\xi^*$ be a locally $D$-optimal subsampling design at $ \beta_{1} $ for a covariate $X_i$ with density $f_{X}$.
	Then, for a covariate $Z_i$ with density $f_{Z}(z) = \frac{1}{|a|} f_{X}(\frac{z - b}{a})$, the design $ \zeta^{*} $ with density $f_{\zeta^*}(z) =  \frac{1}{|a|} f_{\xi^{*}}(\frac{z - b}{a})$ is locally $ D $-optimal at the transformed parameter $ \beta_{1} / a $.
\end{theorem}

For $a = - 1$, Theorem~\ref{theorem:equivariant} covers sign change 
Then we can transfer the characterization of a locally $D$-optimal subsampling design in the equivalence theorem (Theorem~\ref{theo:support}) to positive values for the slope $\beta_{1}$.

\begin{corollary}
	\label{cor:support-positiveslope}
	Let $\beta_{1} > 0$. 
	Then the subsampling design $\xi^*$ 
	is locally $D$-optimal at $ \B $ 
	if and only if
	$f_{\xi^*} = f_X \1_{\XX^*}$
	and either
	\begin{itemize}
		\item[(i)] 
		there exist $ a_1 < a_2 < a_3 $ such that
		\\ 
		$\XX^{*} = [a_1, a_2] \cup [a_3, \infty)$, 
		\\
		$ F_X(a_1) + F_X(a_3) - F_X(a_2) = 1 - \alpha $, and
		\\
		$\psi(a_1, \xi^{*}, \beta_{1}) = \psi(a_2, \xi^{*}, \beta_{1}) = \psi(a_3, \xi^{*}, \beta_{1})$,
		\item[or]
		\item[(ii)]
		$ \XX^{*} = [q_{1 - \alpha}, \infty)$, 
		\\
		$\psi(x, \xi^{*}, \beta_{1}) < \psi(q_{1 - \alpha}, \xi^{*}, \beta_{1})$
		for $x < q_\alpha$, and
		$\psi(x,\xi^{*},\beta_{1}) > \psi(q_{\alpha}, \xi^{*}, \beta_{1})$
		for $x > q_\alpha$.
	\end{itemize}
\end{corollary}

To illustrate how the equivalence theorem (Theorem~\ref{theo:support}) can be used to construct locally $D$-optimal subsampling designs, we consider $\beta_{1} < 0$ in the situation of an exponentially and of a uniformly distributed covariate in the following two examples.

\begin{example}[exponential distribution]
	\label{example:Exp}
	We assume the covariate $X_i$ to follow an exponential distribution with rate $ \lambda $, i.\,e.~$ X_{i} $ has density $ f_{X}(x) = \lambda \exp(- \lambda x) $ for $ x \ge 0 $.
	The condition of finite moments $ m_{k}(\xi, \beta_{1}) $ is satisfied for $ \beta_{1} < \lambda $ and hence, in particular, for $\beta_{1} \leq 0$.
	For $ \beta_{1} < 0 $, let 
	\begin{equation*}
		g_{0}(t) =  
		\frac{\lambda}{\lambda - \beta_{1}} \exp(- (\lambda - \beta_{1}) t),\ 
		g_{1}(t) = 
		\left(t + \frac{1}{\lambda - \beta_{1}}\right) g_{0}(t)
		\mathrm{\ and\ }
		g_{2}(t) = 
		t^2 g_{0}(t) + \frac{2}{\lambda - \beta_{1}} g_{1}(t) 
	\end{equation*}	
	such that $g_{k}(t) = \int_{t}^{\infty} x^{k} \exp(\beta_{1} x) f_{X}(x) \diff x$, $t \geq 0$.
	Then, in scenario~(i), the entries in $ \M(\xi^{*}, \B) $ are 
	\[
	m_{k}(\xi^{*}, \beta_{1}) = g_{k}(0) - g_{k}(a_1) + g_{k}(a_2) - g_{k}(a_3) \, , 
	\quad
	k = 0, 1, 2,
	\]
	while they reduce to
	$ m_{k}(\xi^{*}, \beta_{1}) = g_{k}(0) - g_{k}(q_{\alpha})$ in scenario~(ii) when there is only one interval, where $q_{\alpha} = - \log(1 - \alpha) / \lambda$.
	
	In scenario~(i), we obtain numerical results for the boundary points $a_1$ to $a_3$ solving the system of equations~(\ref{theo:support}a) and (\ref{theo:support}b) using the Newton method implemented in the \textbf{\textsf{R}} package \textit{nleqslv} by \cite{nleqslv}. 
	Note that here $ a_3 < x_{\max} = \infty $.
	For the case of a standard exponential distribution ($\lambda = 1$), results  are given in Table~\ref{Table:Exp} for selected values of $ \beta_{1} $ and $ \alpha $.
	In addition, we give the values for the amount $ F_{X}(a_1) $ as well as the percentage of mass the design $ \xi^{*} $ places on the left interval $ [0, a_1] $.
	We also add the result for $ \beta_{1} = 0 $ for reference \citep[see][]{reuter2023optimal}.
	\begin{table}[h]
		\begin{center}
			\caption{Numerical values for the boundary points $a_1$, $ a_2 $, $a_3$, and $q_{\alpha}$, 
				respectively, for selected values of the subsampling proportion $\alpha$ and slope parameter $ \beta_{1} $
				in the case of a standard exponentially distributed covariate ($\lambda = 1$)}
			\begin{tabular}{ll|ccccc} \toprule
				{$\alpha$} & {$\beta_{1}$} & {$a_1$} & {$a_2$} & {$a_3$, $q_{\alpha}$} & {$F_{X}(a_1)$} 
				& {\% of mass on $[0, a_1]$} \\ \midrule
				\multirow{4}{*}{0.01} & \phantom{-}0.0 & 0.00579 & 5.46588 & -  & 0.00577 & 57.71 
									\\
									& -0.5 & 0.00501 & 3.86767 & \phantom{1}4.14130 & 0.00500 & 49.95 
									\\
									& -1.0 & 0.00500 & 1.98399 & \phantom{1}2.02112 & 0.00499 & 49.88 
									\\
									& -4.0 & 0.00496 & 0.49830 & \phantom{1}0.50665 & 0.00495 & 49.51 
									\\
									\midrule
					\multirow{4}{*}{\textbf{0.10}} & \phantom{-}0.0 & 0.06343 & 3.25596 & -  & 0.06146 		& 61.46 
									\\
									& \textbf{-0.5} & \textbf{0.05181} & \textbf{2.92225} & \phantom{1}\textbf{5.44835} & \textbf{0.05049} & \textbf{50.49} 
									\\
									& -1.0 & 0.05011 & 1.83717 & \phantom{1}2.22435 & 0.04887 & 48.87 
									\\
									& -4.0 & 0.04680 & 0.47740 & \phantom{1}0.56896 & 0.04572 & 45.72 
									\\
									\midrule
				\multirow{4}{*}{0.30} & \phantom{-}0.0 & 0.21398 & 2.23153 & -                  & 0.19264 & 64.21 
									\\
									& -0.5 & 0.17225 & 1.95006 & \phantom{1}7.60885 & 0.15823 & 52.74 
									\\
									& -1.0 & 0.15317 & 1.50902 & \phantom{1}2.76234 & 0.14202 & 47.34 
									\\
									& -4.0 & 0.12876 & 0.40855 & \phantom{1}0.72273 & 0.12081 & 40.27 
									\\
									\midrule
				\multirow{4}{*}{0.75} & \phantom{-}0.0 & 0.67278 & 1.34596 & -                  & 0.48971 & 65.29 
									\\
									& -0.5 & 0.52804 & 1.07947 & 10.89214           & 0.41024 & 54.70 
									\\
									& -1.0 & 0.43176 & 0.88401 & \phantom{1}4.28609 & 0.35063 & 46.75 
									\\
									& -4.0 & - & -       & \phantom{1}1.38629                 	 & - & -  
									\\   		
				\bottomrule
			\end{tabular}	
			\label{Table:Exp}
		\end{center}
	\end{table}
	
	For other values of the rate $\lambda$, results can be derived from the case of a standard exponentially distributed covariate via equivariance (Theorem~\ref{theorem:equivariant}) by letting $a = 1 / \lambda$ and $b = 0$:
	If we seek a locally $D$-optimal subsampling design at $\beta_{1} < 0$ when the rate is $\lambda$, we can first construct a locally $D$-optimal design at $\beta_{1} / \lambda$ for a standard exponentially distributed covariate and then divide the obtained boundary points by $\lambda$.
	For example, when $\lambda = 2$, $\beta_{1} = - 1$ and the subsampling proportion is $\alpha = 0.10$, we get the boundary points $0.05181/2$, $2.92225/2$, and $5.44835/2$ from the second line highlighted in the second block of Table~\ref{Table:Exp} such that the locally $D$-optimal subsampling design wanted is supported on the two intervals $[0, 0.0259]$ and $[1.4611, 2.7242]$.
		
	When the subsampling proportion $ \alpha $ goes to zero, the locally $ D $-optimal subsampling design apparently tends to its counterpart in classical optimal design theory  
	which assigns equal weight $1 / 2$ to two support points $ x_{1}^{*} = 0 $ and $ x_{2}^{*} = - 2 / \beta_{1} $ \cite[see e.\,g.][]{rodriguez2007locally}.
	In particular, we observe $ a_2 < x_{2}^{*} < a_3 $ for all numerically obtained values of $ a_2 $ and $ a_3 $.
	
	On the contrary, we find that scenario~(ii) appears for large values of $ \alpha $.
	This happens when the slope $\beta_{1}$ is strongly negative.
	More precisely, given $ \alpha $, there is a crossover point $ \beta_{1}^{*} $ such that the single interval design with density $f_{\xi^{*}} = f_{X} \1_{[0, q_{\alpha}]}$ is locally $D$-optimal at $\beta_{1}$ for all $\beta_{1} \geq \beta_{1}^{*}$
	This crossover point becomes stronger negative when $ \alpha $ gets smaller and apparently tends to $- \infty$ as $ \alpha \to 0$.
	On the other hand, when $\alpha$ gets larger, the crossover point apparently tends to zero.
	In Table~\ref{Table:ExpCrossover}, we give numerical results for the crossover point $ \beta_{1}^{*} / \lambda$ for selected values of $ \alpha $
	together with the quantile $q_{\alpha}$, the setting $x_{2}^{*}$ of the locally $D$-optimal unbounded design and their ratio.
	This shows that, for scenario~(ii) to apply, the quantile $q_{\alpha}$ has to be substantially larger than $x_{2}^{*}$. 
	Vice versa, for given slope $\beta_{1} < 0$, there is a critical subsampling proportion $\alpha^{*}$ such that  the single interval design is locally $D$-optimal for larger subsampling proportions $\alpha \geq \alpha^{*}$.
	In particular, when $ \beta_{1} = 0 $, only scenario~(i) applies \citep[see][]{reuter2023optimal} and, hence, $\alpha^{*} = 1$.
	
	We further notice that the percentage of mass on the left interval $ [0, a_1] $ is generally larger than $ 50\% $ for $ \beta_{1} $ closer to zero which coincides with what we have seen in \cite{reuter2023optimal} for the case $ \beta_{1} = 0 $.
	There, observations from the right tail are more informative and thus more observations are needed on the left tail. 
	Conversely, the percentage of mass on $ [0, a_1] $ is smaller than $ 50\% $ for strongly negative $ \beta_{1} $.
	Figure~\ref{Figure:Exp} depicts the locally $ D $-optimal subsampling designs for $ \alpha = 0.5, 0.9 $ and $ \beta_{1} = - 1 $ along with the corresponding sensitivity functions.
	The horizontal dotted line represents the threshold $ s^{*} $ from Theorem~\ref{theorem:opt-design}.
	The vertical dotted lines depict the boundary points.	
	While smaller subsampling proportions $ \alpha \le 0.1 $ are typically of interest in the context of subsampling,
	our selection of larger subsampling proportions $ \alpha $ has been made for the sake of clarity and visibility in the tables and figures.

	\begin{figure}[htb]
		\centering
		\begin{subfigure}[t]{.475\textwidth}
			\centering
			\includegraphics[width=\linewidth]{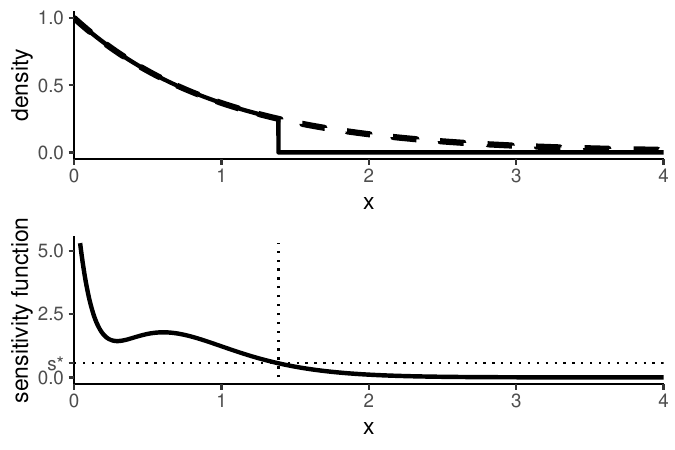}
			\caption{$\beta_{1} = - 4$, $\alpha = 0.75$}
		\end{subfigure}
		\hfill
		\begin{subfigure}[t]{.475\textwidth}
			\centering
			\includegraphics[width=\linewidth]{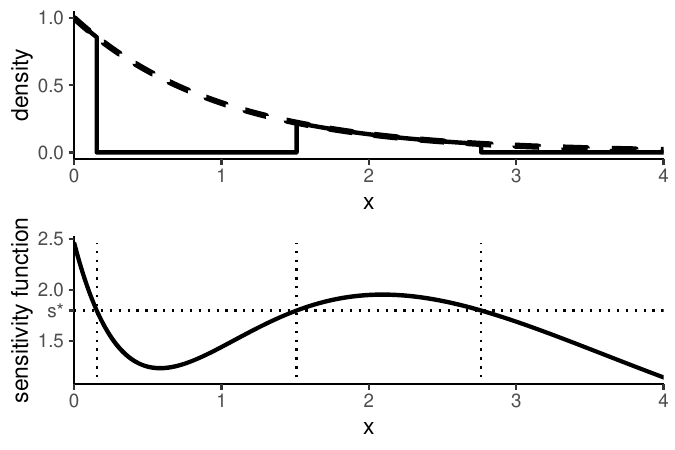}
			\caption{$\beta_{1} = - 1$, $\alpha = 0.3$}
		\end{subfigure}
		\caption{Density of the locally optimal design (solid) at $ \beta_{1} $
			and the standard exponential distribution (dashed, upper panels),
			and corresponding sensitivity functions (lower panels)
			for $\beta_{1} = - 4$, $\alpha = 0.75$ (left) and $\beta_{1} = - 1$, $\alpha = 0.3$ (right)}
		\label{Figure:Exp}
	\end{figure}	
	\begin{table}[h]
		\begin{center}
			\caption{Numerical values for the standardized crossover point $\beta_{1}^{*} / \lambda$  
				for an exponentially distributed covariate}
			\begin{tabular}{c|cccc} \toprule
				
				     {$\alpha$} & 	{$\beta_{1}^{*} / \lambda$} & {$\lambda q_{\alpha}$} & {$\lambda x_{2}^{*}$} & {$q_{\alpha} / x_{2}^{*}$} 
				     \\
				     \midrule
					0.01 & -360.34840 & 0.01005 & 0.00556 & 1.81081
					\\
					0.10 & \phantom{0}-34.60684 & 0.10536 & 0.05779 & 1.82310
					\\
					0.30 & \phantom{0}-10.41165 & 0.35667 & 0.19209 & 1.85679
					\\
					0.50 & \phantom{00}-5.49454 & 0.69314 & 0.36400 & 1.90426
					\\
					0.75 & \phantom{00}-2.89534 & 1.38629 & 0.69077 & 2.00690
					\\
					0.90 & \phantom{00}-1.86128 & 2.30259 & 1.07453 & 2.14288
					\\
				\bottomrule
			\end{tabular}	
			\label{Table:ExpCrossover}
		\end{center}
	\end{table}
\end{example}

\begin{example}[uniform distribution]
	\label{example:Unif}
	We assume the covariate to be uniform random on an interval $ [x_{\min}, x_{\max}] $ with density $ f_{X}(x) = \frac{1}{x_{\max} - x_{\min}} \1_{[x_{\min}, x_{\max}]}(x) $.
	The condition of finite moments $ m_{k}(\xi, \beta_{1}) $ is satisfied for all $ \beta_{1} $.
	
	For $\beta_{1} < 0$, let 
	\begin{equation*}
		g_{0}(t) = 
			\frac{\exp(\beta_{1} t)}{|\beta_{1}| (x_{\max} - x_{\min})} ,\ 
		g_{1}(t) = 
			\left(t + \frac{1}{|\beta_{1}|}\right) g_{0}(t)
		\mathrm{\ and\ }
		g_{2}(t) = 
			t^2 g_{0}(t) + \frac{2}{|\beta_{1}|} g_{1}(t) \, .
	\end{equation*}	
	In scenario~(i), unlike in Example~\ref{example:Exp}, the support of the covariate is bounded from above and thus the rightmost boundary point $ a_3 $ may be larger than $x_{\max}$.
	We denote the essential supremum of $ \xi^{*} $ by $ \tilde{a}_3 = \min(a_3, x_{\max}) $.
	Then, in scenario~(i), the entries in $ \M(\xi^{*}, \B) $ are 
	\[
	m_{k}(\xi^{*}, \beta_{1}) = g_{k}(x_{\min}) - g_{k}(a_1) + g_{k}(a_2) - g_{k}(\tilde{a}_3) \, , 
	\quad
	k = 0, 1, 2,
	\]
	while in scenario~(ii), when there is only one interval, they reduce to
	$ m_{k}(\xi^{*}, \beta_{1}) = g_{k}(x_{\min}) - g_{k}(q_{\alpha})$ where $q_{\alpha} = (1 - \alpha) x_{\min} + \alpha x_{\max}$.
		
	For the case of a uniform distribution on the unit interval ($ x_{\min} = 0 $ and $ x_{\max} = 1 $), optimal boundary points are given in Table~\ref{Table:Unif} for selected values of $ \alpha $ and $ \beta_{1} < 0 $.
	In addition, we give the values for the amount $ F_{X}(a_1) $ as well as the percentage of mass the design $ \xi^{*} $ places on the left interval $ [0, a_1] $.
	We also add formally the result for $ \beta_{1} = 0 $ for reference \citep[see][]{reuter2023optimal}.
		
	\begin{table}[h]
		\begin{center}
			\caption{Numerical values for the boundary points 
				$a_1$, $ a_2 $, $a_3$ and $q_{\alpha}$, respectively, 
				for selected values of the subsampling proportion $\alpha$ and slope parameter $ \beta_{1} $ 
				in the case of a uniformly distributed covariate on $ [0,1] $}
			\begin{tabular}{ll|ccccc} \toprule
				{$\alpha$} & {$\beta_{1}$} & {$a_1$} & {$a_2$} & {$a_3$, $q_{\alpha}$} & {$F_{X}(a_1)$} 
				& {\% of mass on $[0, a_1]$} 
				\\ 
				\midrule
				\multirow{4}{*}{0.01} & \phantom{-}0 & 0.00500 & 0.99500 & - & 0.00500 & 50.00
				\\
				& -2 & 0.00498 & 0.99498 & -       & 0.00498 & 49.75 
				\\
				& -4 & 0.00495 & 0.49994 & 0.50499 & 0.00495 & 49.51 
				\\
				& -8 & 0.00490 & 0.24989 & 0.25498 & 0.00490 & 49.04 
				\\ 
				\midrule
				\multirow{4}{*}{\textbf{0.10}} & \phantom{-}0 & 0.05000 & 0.9500 & - & 0.05000 & 50.00
				\\
				& -2 & 0.04772 & 0.94772 & -       & 0.04772 & 47.72 
				\\
				& \textbf{-4} & \textbf{0.04578} & \textbf{0.49506} & \textbf{0.54928} & \textbf{0.04578} & \textbf{45.78} 
				\\
				& -8 & 0.04269 & 0.24155 & 0.29887 & 0.04269 & 42.69 
				\\ 
				\midrule
				\multirow{4}{*}{0.30} & \phantom{-}0 & 0.15000 & 0.8500 & - & 0.15000 & 50.00
				\\
				& -2 & 0.13271 & 0.83271 & -       & 0.13271 & 44.24 
				\\
				& -4 & 0.12102 & 0.46678 & 0.64577 & 0.12102 & 40.34 
				\\
				& -8 & 0.10847 & 0.20165 & 0.39318 & 0.10847 & 36.16 
				\\ 
				\midrule
				\multirow{4}{*}{0.50} & \phantom{-}0 & 0.25000 & 0.7500 & - & 0.25000 & 50.00
				\\
				& -2 & 0.20993 & 0.70993 & -       & 0.20993 & 41.99 
				\\
				& -4 & 0.18578 & 0.42624 & 0.74046 & 0.18578 & 37.16 
				\\
				& -8 & - & -       & 0.50000       & - & -  
				\\
				\bottomrule
			\end{tabular}	
			\label{Table:Unif}
		\end{center}
	\end{table}
	
	Apart from the situation that $a_3 > x_{\max}$ indicated by a hyphen ($-$) in the table when $\alpha = 0.5$ and $\beta_{1} = - 2$, the results are similar to those in Example~\ref{example:Exp}:
	More weight is given to the left interval $ [0,a] $ when $ \beta_{1} $ is closer to zero.  
	When the subsampling proportion $\alpha$ becomes small, the locally $ D $-optimal subsampling design approaches the locally $ D $-optimal unbounded design equally supported on $ x_1^* = 0 $ and $ x_2^* = - 2 / \beta_{1} $.
	For large values of $\alpha$, the two intervals are merged into one (e.\,g.\ for $\alpha = 0.50$ and $\beta_{1} = - 8$). 
	Figure~\ref{Figure:Unif1} depicts the locally $ D $-optimal subsampling designs along the corresponding sensitivity functions in scenario~(ii) of a single supporting interval for $ \xi^{*} $ in the left panel. 
	The right panel exhibits scenario~(i) of $ \xi^{*} $ supported on two proper intervals. 
	The horizontal dotted line depicts the threshold $ s^{*} $. The vertical dotted lines represent the boundary points $a_1$, $a_2$, and $a_3$.
	The situation when $ a_3 > x_{\max} $ is displayed in Figure~\ref{Figure:Unif2}.
	\begin{figure}[htb]
		\centering
		\begin{subfigure}[t]{.475\textwidth}
			\centering
			\includegraphics[width=\linewidth]{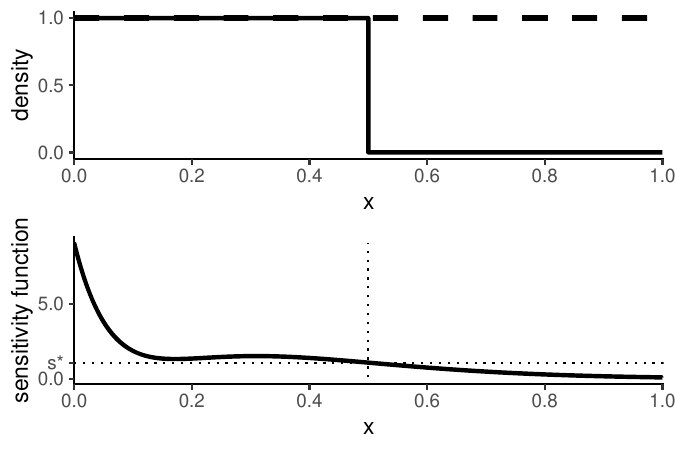}
			\caption{$\beta_{1} = - 8$, $\alpha = 0.5$}
		\end{subfigure}
		\hfill
		\begin{subfigure}[t]{.475\textwidth}
			\centering
			\includegraphics[width=\linewidth]{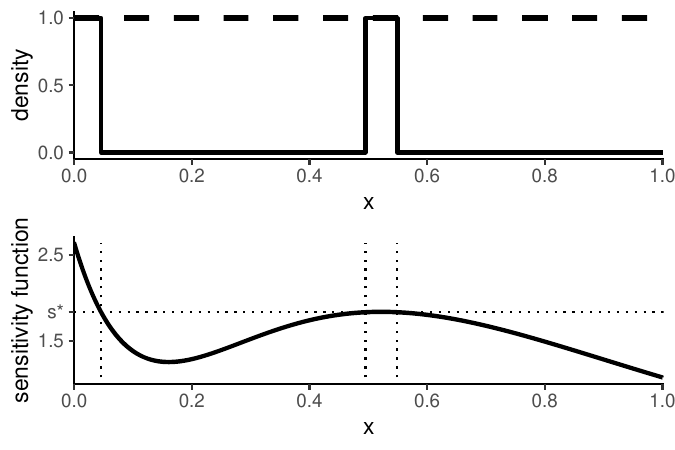}
			\caption{$\beta_{1} = - 4$, $\alpha = 0.1$}
		\end{subfigure}
		\caption{Density of the locally optimal design (solid) at $ \beta_{1}$ 
			for a uniformly distributed covariate on $ [0,1] $ (dashed, upper panels),
			and sensitivity functions (lower panels)
			for $\beta_{1} = - 8$, $\alpha = 0.5$ (left) and $\beta_{1} = - 4$, $\alpha = 0.1$ (right)}
		\label{Figure:Unif1}
	\end{figure}	
	\begin{figure}[htb]
		\centering
		\includegraphics[width=.475\linewidth]{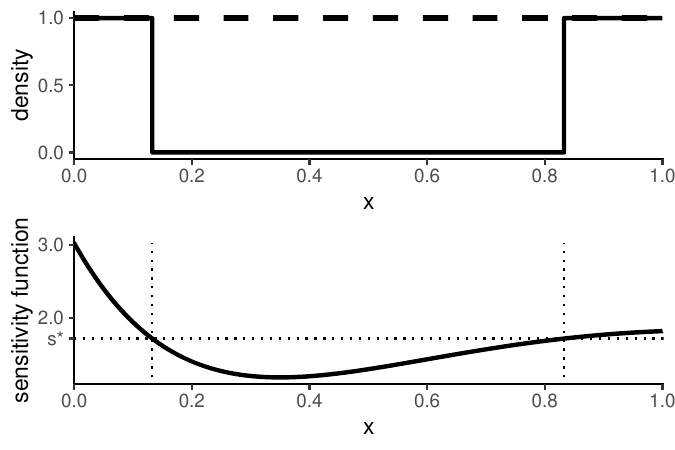}
		\caption{Density of the locally optimal design (solid) at $ \beta_{1}$ 
			for a uniformly distributed covariate on $ [0,1] $ (dashed, upper panel),
			and sensitivity functions (lower panel)
			for $\beta_{1} = - 2$, $\alpha = 0.3$}
		\label{Figure:Unif2}
	\end{figure}
	
	Because of the symmetry of the uniform distribution, locally $D$-optimal subsampling designs can be derived for positive values of the slope $\beta_{1}$	via equivariance with respect to sign change by letting $a = -1$ and $b = 1$ in Theorem~\ref{theorem:equivariant}.
	For example, when $\beta_{1} = 4$ and $\alpha = 0.10$, the optimal boundary points can be obtained from the third line highlighted in the second block of Table~\ref{Table:Unif} as $1 - 0.04578$, $1 - 0.49506$, and $1 - 0.54928$ such that the locally $D$-optimal subsampling design is then supported on the two intervals $[0.45072, 0.50494]$ and $[0.95422, 1]$.
	
	Further, for other ranges $[x_{\min}, x_{\max}]$ of the uniform covariate, optimal subsampling designs can be obtained by equivariance (Theorem~\ref{theorem:equivariant}) as well by letting $a = x_{\max} - x_{\min}$ and $b = x_{\min}$.
\end{example}

\section{Efficiency}
\label{sec:efficiency}
We want to study the performance of random subsampling as well as some heuristic subsampling designs in the style of IBOSS \cite[see][]{wang2019information} to quantify the gain in using a locally $ D $-optimal subsampling design.
Besides, we are interested in the quality of the heuristic designs and how they compare to random subsampling. 
Further, we want to investigate the performance of designs when the parameter is misspecified. 
Specifically, a subsampling design $ \xi^{*}(\B') = \argmax \det(\M(\xi,\B'))$ that is locally $ D $-optimal at $ \B' $ is studied when the true parameter is $ \B $.  
The performance of a design $ \xi $ may be compared to the locally $ D $-optimal subsampling design $ \xi^{*}(\B) $ using $ D $-efficiency.
The $D$-efficiency of a subsampling design $ \xi $ with mass $\alpha$ is defined as
\begin{equation*}
	\label{eq:efficiency}
	\eff_{D,\alpha}(\xi,\B) = \left(\frac{\det(\M(\xi,\B))}{\det(\M(\xi^{*}(\B),\B))}\right)^{1/2}.
\end{equation*}
For this definition the homogeneous version $(\det(\M(\xi,\B)))^{1/2}$
of the $D$-criterion is used which satisfies the homogeneity condition
$(\det(\nu \M))^{1/2} = \nu (\det(\M))^{1/2}$
for all $\nu > 0$
\cite[see][Chapter~6.2]{pukelsheim1993optimal}.
Note that by Remark~\ref{remark:beta0}, the efficiency $ \eff_{D,\alpha}(\xi,\B) $ does not depend on $ \beta_{0} $. 

As uniform random subsampling we define the design $\xi_{\alpha}$ of size $\alpha$,
which has density $f_{\xi_{\alpha}}(x) = \alpha f_X(x)$.
The information matrix of $ \xi_{\alpha} $ is given by
$ \M(\xi_{\alpha},\B) = \alpha \M(\xi_{1},\B) $.
Here, $ \xi_{1} $ represents the full sample with information matrix
$ \M(\xi_{1},\B) =  \int \exp(\beta_{0} + \beta_{1} x) \f(x)\f(x)^{\top} f_{X}(x)\diff x $.
Thus, the $D$-efficiency $\eff_{D,\alpha}(\xi_{\alpha},\B)$ of uniform random subsampling
can be nicely interpreted as noted in \cite{reuter2023optimal}:
for a fixed full sample size $ n $, the required subsample size (mass) $ \tilde{\alpha} $ needed to achieve the same precision (measured by the $D$-criterion), 
compared to utilizing a locally $D$-optimal subsampling design $\xi^*$ with mass $\alpha$,
is given by the inverse of the efficiency, $\eff_{D,\alpha}(\xi_{\alpha},\B)^{-1}$, multiplied by $\alpha$, i.\,e.~$ \tilde{\alpha} = \alpha/ \eff_{D,\alpha}(\xi_{\alpha},\B) $.
For instance, if the efficiency $\eff_{D,\alpha}(\xi_{\alpha},\B)$ equals
$0.5$, then twice the number of observations would be needed under 
uniform random sampling compared to a locally $D$-optimal subsampling design of mass $\alpha$.
Naturally, the full sample has higher information than any proper subsample
such that, for uniform random subsampling, 
$\eff_{D,\alpha}(\xi_{\alpha},\B) \geq \alpha$ holds for all $\alpha$.

Further, we analyze the efficiency of two heuristic designs.
Again we only consider the case $ \beta_{1} < 1 $.
Let the $ \alpha $-quantile of the covariate $ X_{i} $ be denoted by $ q_{\alpha} $. 
First, we consider the one-sided design $ \xi_{os} $ with density 
$ f_{\xi_{os}}(x) = f_{X}(x)\1_{(-\infty,q_{\alpha}]}(x) $ that assigns all of its mass on the left tail of the distribution of the covariate motivated by its optimality for large $\alpha$.
Second, we study the two-sided design $ \xi_{ts} $ with density
$ f_{\xi_{ts}}(x) = f_{X}(x)\1_{(-\infty,q_{\alpha/2}]\cup [q_{1-\alpha/2},\infty)}(x) $ that allocates equal mass $\alpha / 2$ on both tails of the distribution in the style of the IBOSS method \cite[see][]{wang2019information}. 

\begin{example}[exponential distribution]
	\label{example:Exp-Eff}
	As in Example~\ref{example:Exp}, we assume that the covariate $ X_{i} $ is exponentially distributed with rate $ \lambda $.
		
	Because uniform random subsampling $ \xi_{\alpha} $ as well as the one- and two-sided designs $ \xi_{os} $ and $ \xi_{ts} $ are equivariant under location-scale transformations, their efficiency depends only on the slope and the rate by the ratio $ \beta_{1} / \lambda $.
	In Figure~\ref{Figure:Exp-Eff}, we depict the efficiency of these designs  for $ \beta_{1} / \lambda = - 1 $ and $ - 4 $ in dependence on the subsampling proportion $ \alpha $.
	The efficiency of uniform random subsampling is quite low for reasonable proportions $ \alpha \le 0.1 $ and, hence, the gain in using the D-optimal subsampling design is substantial.
	Similarly, the efficiency of the one- and the two-sided design is small for $ \alpha \le 0.1 $ and apparently tends to zero for $\alpha \to 0$ which may be explained by the fact that these designs miss observations close to the location $x_{2}^{*}$ of the locally $D$-optimal unbounded design.
	This feature does not apply to uniform random subsampling such that, for very small subsampling proportions, both the one- and the two-sided design is severely less efficient than uniform random subsampling. 	
	
	As is to be expected, the two-sided IBOSS-like design $ \xi_{ts} $ performs much better for $ \beta_{1} $ near zero. 
	In particular, for $ \beta_{1} = 0 $, the two-sided design $ \xi_{ts} $ only differs slightly from the locally $ D $-optimal subsampling design is $ \xi^{*}$ and has a high efficiency throughout \citep[see][]{reuter2023optimal}.
	Conversely, the one-sided design $ \xi_{os} $ performs better for strongly negative $ \beta_{1} $.
	The vertical dotted line in Figure~\ref{Figure:Exp-Eff} displays the crossover point $ \alpha^{*} $.
	For all $ \alpha > \alpha^{*} $, the one-sided design is the $ D $-optimal subsampling design.
	
	We observe similar behavior in Figure~\ref{Figure:Exp-Eff-betaXAxis}.
	Predictably, the one-sided design performs better for strongly negative $ \beta_{1} $ and the two-sided design is better for $ \beta_{1} $ closer to zero.
	Notably, the two-sided design exhibits a nonmonotonic behavior: It performs worst for $ \beta_{1} / \lambda \approx -3.64 $ ($ \eff_{D,\alpha}(\xi_{ts},\B) = 0.07974506 $) and attains a local maximum at $ \beta_{1} / \lambda \approx -0.40 $ ($ \eff_{D,\alpha}(\xi_{ts},\B) = 0.9988009 $).
	Further, we again see that uniform subsampling generally performs better for $ \beta_{1} $ closer to zero, though it performs best for $ \beta_{1} / \lambda \approx -1.05 $ ($ \eff_{D,\alpha}(\xi_{ts},\B) = 0.6978610 $).
	\begin{figure}[htb]
		\centering
		\begin{subfigure}[t]{.475\textwidth}
			\centering
			\includegraphics[width=\linewidth]{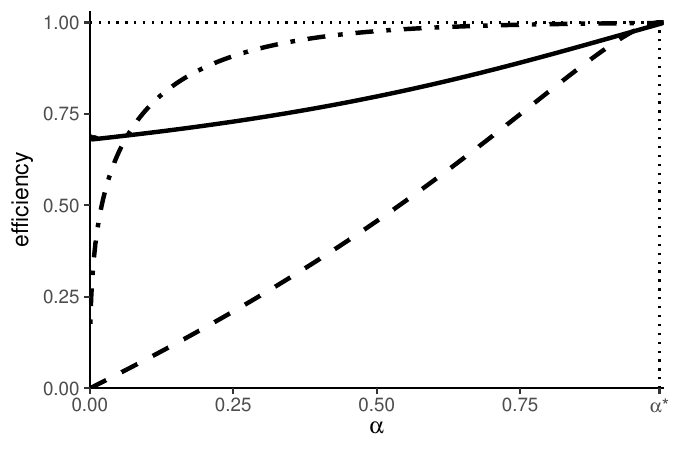}
			\caption{$\beta_{1} / \lambda = - 1$}
		\end{subfigure}
		\hfill
		\begin{subfigure}[t]{.475\textwidth}
			\centering
			\includegraphics[width=\linewidth]{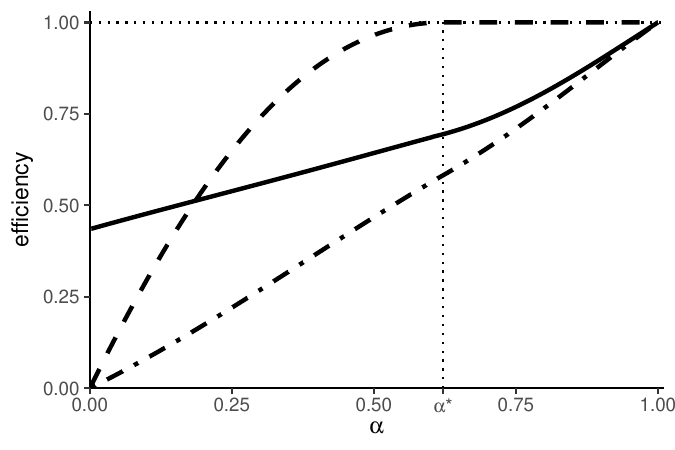}
			\caption{$\beta_{1} / \lambda = - 4$.}
		\end{subfigure}
		\caption{$ D $-efficiency of uniform random subsampling (solid), one-sided (dashed), and two-sided (dot-dashed) subsampling design in dependence on the subsampling proportion $\alpha$ for slope-rate ratio $ \beta_{1} / \lambda = - 1 $ (left) and $ - 4 $ (right) for an exponentially distributed covariate}
		\label{Figure:Exp-Eff}
	\end{figure}

	\begin{figure}[htb]
		\centering
		\includegraphics[width=.475\linewidth]{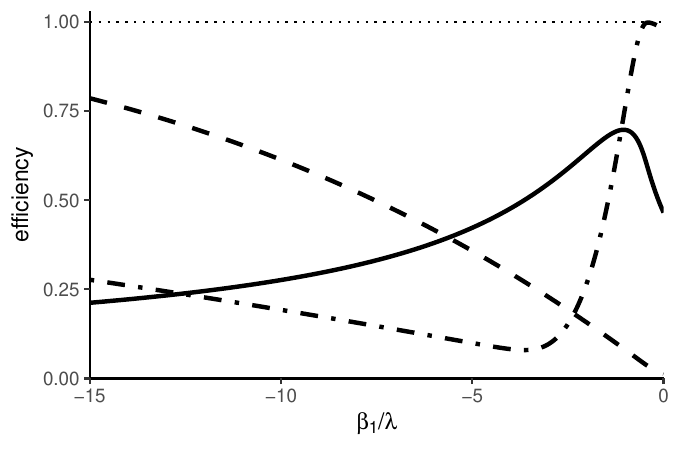}
		\caption{$ D $-efficiency of uniform random subsampling (solid), one-sided (dashed), and two-sided (dot-dashed) subsampling design in dependence on the slope-rate ratio $ \beta_{1} / \lambda $ for subsampling proportion $\alpha = 0.1 $ and an exponentially distributed covariate}
		\label{Figure:Exp-Eff-betaXAxis}
	\end{figure}

	For strongly negative $ \beta_{1} $, the behavior of the efficiency of the three designs in Figure~\ref{Figure:Exp-Eff-betaXAxis} gives additional insight. 
	As $ \beta_{1} \to -\infty $, the efficiency of uniform random subsampling converges to its lower bound $ \alpha $ whereas the efficiency of both one- and two-sided design converge to one. 
	Most of the information is concentrated on the covariate values close to zero. 
	Thus, for strongly negative $ \beta_{1} $ the two heuristic designs as well as the $ D $-optimal subsampling design have almost all the information of the full sample. 
	This limiting behavior is not presented in Figure~\ref{Figure:Exp-Eff-betaXAxis} in order to preserve intelligibility for $ \beta_{1} $ closer to zero.
	
	Finally, we consider the efficiency of locally $ D $-optimal subsampling designs $ \xi^{*}(\B^{\prime}) $, when the nominal value $ \beta_{1}^{\prime} $ is misspecified and differs form the true slope $ \beta_{1} $.
	The left panel of Figure~\ref{Figure:Exp-misspecified-Eff} illustrates the efficiency of $ \xi^{*}(\B^{\prime}) $ in dependence on the subsampling proportion $ \alpha $ for selected values of the true ratio $ \beta_{1} / \lambda $, when the nominal value is $ \beta_{1}^{\prime} / \lambda = - 1 $.
	For all values we find that the efficiency of the design $ \xi^*(\B^{\prime}) $ under misspecification declines with decreasing $ \alpha $.  
	When the deviation of the parameter is rather small, $ \beta_{1} / \lambda = - 0.8$ and $ \beta_{1} / \lambda = - 1.2$, the designs under misspecification are still very efficient, with efficiency above $ 0.98 $ for $ \alpha = 0.01 $.
	For larger deviations however, the efficiency can drop drastically. 
	In particular, when $ \beta_{1} / \lambda $ is closer to $ 0 $, the efficiency is more strongly negatively affected than when the deviation of $ \beta_{1} / \lambda $ is away from zero. 
	In the right panel of Figure~\ref{Figure:Exp-misspecified-Eff}, we exhibit the efficiency for various values of the nominal slope-rate ratio in dependence on the true value when the subsampling proportion is $\alpha = 0.1$. 
	The nominal values are indicated by vertical dotted lines.
	
	It can be seen that the efficiency decreases faster for $ \beta_{1} / \lambda $ towards zero than for stronger negative values.
	In particular, the efficiency increases again when $ \beta_{1} / \lambda $ goes to $- \infty$.
		
	\begin{figure}[htb]
		\centering
		\begin{subfigure}[t]{.475\textwidth}
			\centering
			\includegraphics[width=\linewidth]{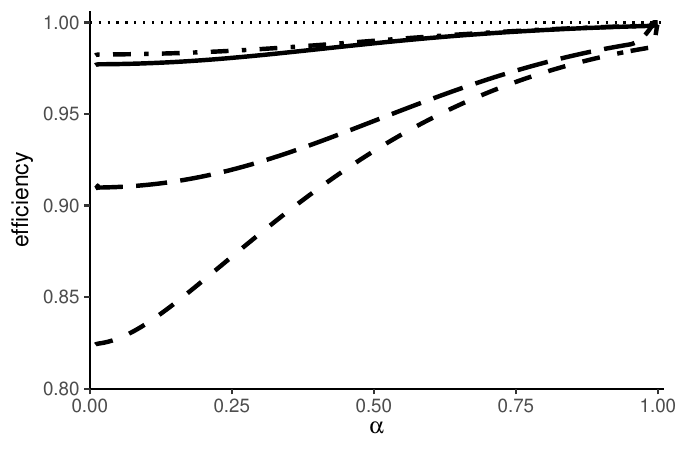}
			\caption{True parameter $ \beta_{1} / \lambda = - 0.5$ (dashed), $ - 0.8$ (solid), $ - 1.2$ (dot-dashed), and $ - 1.5$ (long dashed)}
		\end{subfigure}
		\hfill
		\begin{subfigure}[t]{.475\textwidth}
			\centering
				\includegraphics[width=\linewidth]{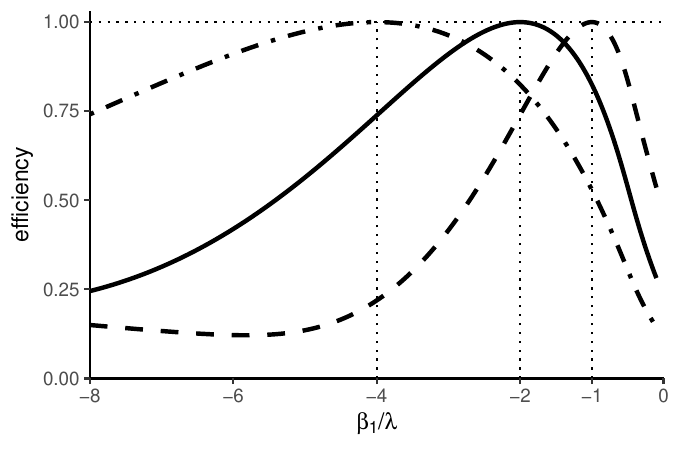}
			\caption{Locally $D$-optimal subsampling designs for $ \beta_{1}^{\prime} / \lambda = - 1$ (dashed), $- 2$ (solid), and $- 4$ (dot-dashed)}
		\end{subfigure}
		\caption{Efficiency of the locally $ D $-optimal subsampling design for $ \beta_{1}^{\prime} / \lambda = - 1$ and various subsampling proportions $ \alpha $ (left) and for subsampling proportions $ \alpha = 0.1$ and various values  of the nominal slope-rate ratio $ \beta_{1}^{\prime} / \lambda $ (right) in dependence on the true slope-rate ratio $ \beta_1 / \lambda $ for an exponentially distributed covariate}
		\label{Figure:Exp-misspecified-Eff}
	\end{figure}
\end{example}

\section{Concluding Remarks}
\label{sec:discussion}
Our investigation centers on a theoretical approach to evaluate subsampling designs under distributional assumptions on the covariate in the case of Poisson regression on a single covariate.
We adjust a standard equivalence theorem to Poisson regression, given a general distribution of the covariate and negative slope parameter $ \beta_{1} $.
This equivalence theorem also characterizes the support of the locally $ D $-optimal subsampling design and allows us to derive such designs for a given covariate and slope parameter.
Then, we establish a theorem to identify locally $ D $-optimal subsampling designs under a scale-location transformation of the covariate and simultaneous rescaling of the slope parameter.
We make use of this to give a corollary to the equivalence theorem for $ \beta_{1} > 0 $. 
It is worthwhile noting that many of the results can be extended from $ D $-optimality to other criteria within Kiefer's class of $ \Phi_{q} $-optimality criteria, including, in particular, linear criteria
The derivation relies mostly on the fact that the sensitivity function can be factorized into the exponential function and a quadratic polynomial, rather than its specific form.
Our efficiency analysis shows, among other things, that heuristic one- or two-sided designs can be highly efficient under certain circumstances, however, they display substantial loss in efficiency for the most relevant small subsampling proportions.
Addressing uncertainty about the parameter $ \beta_{1} $ and the covariate distribution may involve an initial random subsampling step, before deploying the locally $ D $-optimal subsampling design. 
Lastly, note that the results presented here may be extended to polynomial Poisson regression, where the linear predictor is a polynomial of degree $ q $ in the covariate $ X_{i} $. 
Then, the equation $ \psi(x,\xi,\B) = s $ has at most $ 2q + 1 $ solutions and the support of $ \xi^{*} $ is the union of at most $ q+1 $ intervals.

\section*{Acknowledgments}
The work of the first author is supported by the Deutsche Forschungsgemeinschaft (DFG, German Research Foundation) - 314838170, GRK 2297 MathCoRe.

\appendix 
\section{Proofs}

Before we establish the equivalence theorem (Theorem~\ref{theo:support}),
we introduce some technical tools:
The directional derivative of the $ D $-criterion at design $ \xi $ in the direction of a design $ \eta $ is
$ \Psi(\xi,\eta,\B) = \trace(\M(\xi,\B)^{-1}\M(\eta,\B)) - 2 $.
Here, $ \eta $ may be any design of total mass $ \alpha $ which is not necessarily required to have a density bounded by $ f_X $.
The sensitivity function $ \psi(x,\xi,\B) =  \trace(\M(\xi,\B)^{-1}\M(\xi_{x},\B)) $ is the essential part of the directional derivative at $ \xi $ in the direction of a single point design $ \xi_x $ with all mass $ \alpha $ at point $ x $.
Then
\begin{align*}
	\psi(x,\xi,\B) 
	&= \alpha \exp(\beta_{0} + \beta_{1} x) \f(x)^{\top} \M(\xi,\B)^{-1} \f(x) \\
	&= \frac{\alpha}{d(\xi,\beta_{1})} \exp(\beta_{1} x)
	(m_{0}(\xi,\beta_{1}) x^2 - 2 m_{1}(\xi,\beta_{1}) x + m_{2}(\xi,\beta_{1})) 
\end{align*}
does not depend on $\beta_{0}$ and will be denoted by $\psi(x,\xi,\beta_{1})$, for short.
Note that, for any continuous subsampling design $ \xi $, the information matrix $ \M(\xi,\B) $ is positive definite and, hence, $ \psi(x,\xi,\beta_{1})$ is well-defined. 

For convenience, we reproduce an equivalence theorem for subsampling designs in a general model context which follows from Corollary~1(c) in \cite{sahm2001note}.
\begin{theorem}
	\label{theorem:opt-design}
	Let condition
	\\
	(A) $\Prob(\psi(X_{i}, \xi, \beta_{1}) = s) = 0$ for any $\xi$ and $s$
	\\
	be satisfied.
	Then the subsampling design $\xi^*$ is locally $D$-optimal at $ \B $ if and only if
	there exist a set $ \XX^{*} $ and a threshold $ s^* $ such that
	\begin{itemize}
		\item[(i)] $ \xi^{*} $ has density $ f_{\xi^{*}}(x) = f_{X}(x) \1_{\XX^{*}}(x) $
		\item[(ii)] $\psi(x,\xi^{*},\beta_{1}) \geq s^*$ for $x \in \XX^*$, and
		\item[(iii)] $\psi(x,\xi^{*},\beta_{1}) < s^*$ for $x \not\in \XX^*$.
	\end{itemize}	
\end{theorem}

Next we establish that condition~(A) holds for the Poisson regression model. 
\begin{lemma}
	\label{Lemma:roots}
	Given $\xi$ and $s$, the equation $ \psi(x,\xi,\beta_{1}) = s $ has, at most, three different solutions in $x$. 
\end{lemma}
\begin{proof} 
	For $ \beta_{1} = 0 $, the sensitivity function is a quadratic polynomial in $ x $.
	Hence, there are, at most, two solutions.
	
	For $ \beta_{1} \neq 0 $, the sensitivity function $ \psi(x,\xi,\B) = \exp(\beta_1 x) q(x)$ factorizes into the exponential function ($\exp(\beta_1 x)$) and a quadratic polynomial $ q $ with positive leading term.
	Because $ \psi(x,\xi,\beta_{1}) $ is positive,	only $ s > 0 $ has to be considered.
	Let $ v(x) = q(x) - s \exp(- \beta_{1} x) $.
	The third derivative $ v^{(3)}(x) = s \beta_{1}^{3} \exp(- \beta_{1} x)$ has no roots.
	By iterative application of the mean value theorem, we see that $ v $ has, at most, three roots. 
	Because the solutions of $ \psi(x,\xi,\beta_{1}) = s $ are the roots of $ v $, this completes the proof. 
\end{proof}

Condition~(A) follows from the continuous distribution of the covariate $X_{i}$.

\begin{proof}[Proof of Theorem~\ref{theo:support}]
	If $\xi^*$ is locally $D$-optimal, then, by Theorem~\ref{theorem:opt-design}, its density has the shape $f_{\xi} = f_{X} \1_{\XX}$ and $\XX^{*} = \{x;\, \psi(x,\xi^{*},\beta_{1}) \geq s^{*}\}$ for some $s^{*} > 0$. 
	Because $ \beta_{1} < 0 $, the sensitivity function $ \psi(x, \xi^*,\beta_{1})$ ranges from $ \infty $ for $ x \to - \infty $ to $0 $ for $ x \to \infty $ with $ \psi(x, \xi^*,\beta_{1}) > 0 $ throughout.
	Thus, the number of sign changes in $\psi(x, \xi^{*}, \beta_{1}) - s^{*}$ is odd and, by Lemma~\ref{Lemma:roots}, equal to one or three.
	Hence, $\XX^{*}$ consists of one or two intervals including a left open interval $(- \infty, a_1]$, say, and potentially a second finite interval $ [a_2, a_3] $.
	Conditions (\ref{theo:support}a) and (\ref{theo:support}a'), respectively, follow from the subsampling percentage $\alpha$.
	If there are two intervals, then $\psi(a_k, \xi^{*}, \beta_{1}) = s^{*}$, $k = 1, 2, 3$, by continuity of the sensitivity function and we get condition~(\ref{theo:support}b) in scenario~(i).
	If there is only one interval, then condition~(\ref{theo:support}b') follows from (ii) and (iii) in Theorem~\ref{theorem:opt-design} which completes the proof that the locally $D$-optimal subsampling design satisfies the properties stated in Theorem~\ref{theo:support}.
	
	Conversely, by the shape of the sensitivity function, the properties stated in Theorem~\ref{theo:support} imply the equivalence conditions in Theorem~\ref{theorem:opt-design} which proves the reverse statement.
\end{proof}

\begin{proof}[Proof of  Remark~\ref{rem:left-interval}]
	Assume $ a_1 \leq x_{\min} $. 
	Then 
	\begin{equation*}
		m_1(\xi^{*},\beta_{1}) 
		= \int_{a_2}^{a_3} x \exp(\beta_{1} x) f_{X}(x)\diff x 
		> a_2 \int_{a_2}^{a_3} \exp(\beta_{1} x) f_{X}(x)\diff x 
		= a_2 m_0(\xi^{*},\beta_{1})
	\end{equation*} 
	and $ q $ attains its minimum at $m_1(\xi^{*},\beta_{1}) / m_0(\xi^{*},\beta_{1}) > a_2 $.
	Hence, the sensitivity function $ \psi(x,\xi^{*},\beta_{1}) =\exp(\beta_{1} x) q(x) $ is strictly decreasing on $ (- \infty, a_2] $ such that  $ \psi(a_1,\xi^{*},\beta_{1}) > \psi(a_2,\xi^{*},\beta_{1})$ which leads to a contradiction to the optimality condition~(\ref{theo:support}b).
\end{proof}

\begin{proof}[Proof of Theorem~\ref{theorem:equivariant}]
	The proof goes along the same lines as in \cite{radloff2016invariance}.
	Denote by $g$ the location-scale transformation $g(x) =a x + b$.	
	Let $Z_{i} = g(X_{i})$.
	Note that only the distribution of the covariate plays a role, but not the covariate itself.
	The transformation $ g $ is conformable with the regression function $ \f(x) $, i.\,e.~there exists a nonsingular matrix 
	$ \mathbf{Q} =  \begin{pmatrix}
		1 & 0 \\ b & a
	\end{pmatrix} $
	such that $ \f(a x + b) = \mathbf{Q} \f(x) $ for all $x$. 
	For a design $ \xi $ bounded by $ f_{X} $, 	we define the transformed design $ \zeta = \xi^{g }$ which has density $  f_{\zeta}(z) = \frac{1}{|a|} f_{\xi}(\frac{z - b}{a}) $ and is bounded by the density $  f_{Z}(z) = \frac{1}{|a|} f_{X}(\frac{z - b}{a}) $ of $Z_{i}$.
	Further, let $\tilde{\B} = (\mathbf{Q}^{\top})^{-1} \B = (\beta_{0} - \beta_{1}b/a, \beta_{1}/a)^{\top}$.
	By the transformation theorem for measure integrals,
	\begin{align*}
		\M(\zeta,\tilde{\B}) 
		&= \int \exp(\beta_{0} + \beta_{1} (z - b) / a) \f(z) \f(z)^{\top} \zeta(\diff z) \\
		&= \int \exp(\beta_{0} + \beta_{1} x) \mathbf{Q} \f(x) \f(x)^{\top} \mathbf{Q}^{\top} \xi(\diff x) \\
		&= \mathbf{Q} \M(\xi,\B) \mathbf{Q}^{\top}.
	\end{align*}
	Therefore 
	$ \det(\M(\zeta,\tilde{\B})) = \det(\mathbf{Q})^{2} \det(\M(\xi,\B))$. 
	Thus $ \xi^{*} $ maximizes the $ D $-criterion over the set of 
	subsampling designs bounded by $ f_{X} $ for $ \beta_{1} $ if and only if $ \zeta^{*} $ 
	maximizes the $ D $-criterion over the set of subsampling designs bounded 
	by $ f_{Z} $ for $ \beta_{1}/a $.
\end{proof}

\bibliographystyle{plainnat}  
\bibliography{refPoisson}

\end{document}